\documentclass[12pt]{article}
\usepackage{url}
\usepackage{amsmath}
\usepackage{amssymb}
\usepackage{amsthm}
\usepackage{algpseudocode} % algorithmicx
\usepackage[section]{algorithm} % put algorithmic in a floating object
\usepackage{bm}
\usepackage[short]{optidef}
\usepackage{color}
\usepackage{enumitem}
\renewcommand{\b}{{\bm{b}}}
\renewcommand{\k}{{\bm{k}}}
\renewcommand{\c}{{\bm{c}}}

\renewcommand{\S}{\mathcal{S}}
\newcommand{\e}{{\bm{e}}}
\newcommand{\R}{{\mathbb{R}}}
\renewcommand{\r}{{\bm{r}}}
\newcommand{\s}{{\bm{s}}}

\renewcommand{\L}{\mathcal{L}}
\newcommand{\X}{\mathcal{X}}

\newcommand{\Y}{\mathcal{Y}}
\newcommand{\w}{{\bm{w}}}
\newcommand{\x}{{\bm{x}}}
\newcommand{\y}{{\bm{y}}}
\newcommand{\z}{{\bm{z}}}
\newcommand{\bz}{{\bm{0}}}
\newcommand{\T}{\top}
\newcommand{\bigO}{{\mathcal{O}}}

\newcommand{\sgn}{{\mathop{\mathrm{sgn}}}}
\newcommand{\proj}{{\mathrm{proj}}}

\newcommand{\argmin}{{\mathop{\mathrm{argmin}}}}
\newcommand{\argmax}{{\mathop{\mathrm{argmax}}}}
\newtheorem{lemma}{Lemma}
\newtheorem{theorem}{Theorem}

\newcommand{\eps}{\epsilon}
\newcommand{\algorithmicbreak}{\textbf{break}}
\newcommand{\Break}{\State \algorithmicbreak}
\newcommand{\be}{\begin{equation}}
\newcommand{\ee}{\end{equation}}

\title{A Primal-Dual Frank-Wolfe Algorithm for Linear Programming\thanks{This research was supported in part by a Discovery Grant from the Natural Sciences and Engineering Research Council (NSERC) of Canada.}}
\author{Matthew Hough\thanks{Department of Combinatorics \& Optimization,
University of Waterloo, 200 University Ave.~W., Waterloo, ON, N2L 3G1,
Canada, {\tt mhough@uwaterloo.ca}.} \and 
Stephen A.~Vavasis\thanks{Department of Combinatorics \& Optimization,
University of Waterloo, 200 University Ave.~W., Waterloo, ON, N2L 3G1,
Canada, {\tt vavasis@uwaterloo.ca.}}}

\begin{document}
\maketitle
\begin{abstract}
  We present two first-order primal-dual algorithms for solving saddle point formulations of linear programs, namely
  FWLP (Frank-Wolfe Linear Programming) and FWLP-P. The former iteratively applies the Frank-Wolfe algorithm to both
  the primal and dual of the saddle point formulation of a standard-form LP. The latter is a modification of FWLP in
  which regularizing perturbations are used in computing the iterates. We show that FWLP-P converges to a primal-dual
  solution with error $\bigO(1/\sqrt{k})$ after $k$ iterations, while no convergence guarantees are provided for FWLP.
  We also discuss the advantages of using FWLP and FWLP-P for solving very large LPs.
  In particular, we argue that only part of the matrix $A$ is needed at each iteration, in contrast to other first-order methods.
\end{abstract}

\section{Introduction}
In recent years, data science applications have given birth to problems of very large scale. This poses a problem for mature LP solvers that require solving a system of linear equations at each iteration.
First-order methods (FoMs) for linear programming aim to solve LPs in such a way that their most expensive operation at each iteration is the product of a matrix and a vector.
Their goal is to provide an alternative to the practitioner over LP algorithms such as the simplex method or interior point methods for large-scale problems.

The Frank-Wolfe algorithm \cite{FrankWolfe1956}, also referred to as the conditional gradient algorithm, is a FoM for minimizing a smooth convex objective function
over a compact convex set. A major benefit of the Frank-Wolfe algorithm is that each iteration requires only the solution of a linear optimization problem over
a convex constraint set, a problem which can be solved efficiently over many constraint sets used in practice. It is known that the Frank-Wolfe algorithm converges at a rate of
$\bigO(1/k)$ \cite{Jaggi2013}.

The focus of this paper is on finding optimal solutions to linear programs in standard form, that is, solving the following optimization problem:
\begin{mini} 
{} { \c^T\x }{ \label{eq:std-lp} }{} 
\addConstraint{ A\x }{ = \b }{}
\addConstraint{ \x }{ \geq \bz }{},
\end{mini}
where we assume throughout that \eqref{eq:std-lp} has an optimal solution. The dual linear program associated with \eqref{eq:std-lp} is
\begin{maxi} 
{} { \b^T\y }{ \label{eq:std-dual} }{} 
\addConstraint{ A^T\y }{ \leq \c }{}.
\end{maxi}

We propose two first-order primal-dual algorithms for simultaneously solving \eqref{eq:std-lp} and \eqref{eq:std-dual} inspired by the Frank-Wolfe algorithm but using the non-standard step-size of $1/(k+1)$ analyzed in \cite{Freund2016}. We call our algorithms FWLP and FWLP-P.
FWLP, first introduced in \cite{Hough2023}, is derived from iteratively applying the Frank-Wolfe algorithm to the primal and dual
problems of a modified saddle-point formulation of \eqref{eq:std-lp}:
\begin{equation} \label{eq:fwlp_saddle}
  \min_{\x\in\Delta}\max_{\y\in\Gamma}\L(\x,\y) := \c^T\x + \y^T(\b - A\x),
\end{equation}
where we define 
\be\label{eq:DeltaGamma}
\Delta = \{\x \in \R^n : \x \geq \bz, \e^T\x \leq \xi\}\mbox{ and }\Gamma = [-\eta,\eta]^m.
\ee
Let $(\x^*,\y^*)$ denote an optimal primal-dual pair of solutions to \eqref{eq:std-lp} and \eqref{eq:std-dual}. The parameters $\xi,\eta > 0$ are assumed to be chosen large enough to ensure that $2\xi \geq \lVert\x^*\rVert_1$ and
$2\eta \geq \lVert\y^*\rVert_{\infty}$, thus describing redundant constraints. Despite their redundance, these constraints
are necessary for compactness of the feasible sets corresponding to the primal and dual subproblems solved by FWLP.
FWLP-P is a modification of the original FWLP algorithm in which regularizing perturbations are introduced when computing the iterates.
FWLP-P is of theoretical importance to FWLP since we are able to prove the convergence of FWLP-P but no convergence proof is known yet for FWLP.

With the ubiquity of very large-scale problems, FoMs for linear programming have seen interest of late \cite{PDLP1,PDLP2,ECLIPSE2020,SNIPAL2020}.
Such methods typically apply existing ideas and algorithms from continuous optimization to linear programming.
Our work is related particularly to \cite{Gidel2017}, where Gidel et al.\
use the Frank-Wolfe algorithm to solve general convex-concave saddle point problems:
\begin{equation} \label{eq:general_saddle}
  \min_{\x \in \X}\max_{\y \in \Y} f(\x,\y),
\end{equation}
where $f$ is a smooth convex-concave function and $\X\times\Y$ is a convex compact set. In analyzing the convergence of their method,
the authors introduce a potential function bounding the distance to optimality and then show that this potential function
decreases to zero at a given rate. We employ this technique in Section~\ref{sec:convergence} to analyze the convergence of FWLP-P.  The results of \cite{Gidel2017} do not directly apply to \eqref{eq:fwlp_saddle} because their strong-convexity assumption is not satisfied.
The FWLP algorithm is also closely related to the Generalized Fictitious Play Algorithm proposed by Hammond in her 1984 PhD thesis \cite{Hammond1984}.
Applying Hammond's algorithm to \eqref{eq:fwlp_saddle} yields an algorithm very similar to FWLP, the only difference being that
the update of $\y_k$ is based on $\x_k$ instead of $\x_{k+1}$ like in FWLP. For a more in-depth analysis of the similarities
between FWLP and Hammond's Generalized Fictitious Play, see \cite[Chapter 6.2]{Hough2023}.

The algorithms FWLP % \cite{Hough2023}
and FWLP-P are described in Section~\ref{sec:algorithms}.  Iterations of FWLP are simpler and faster than those of FWLP-P, but we do not have a convergence proof of FWLP.  Our preliminary computational tests (not reported here) indicate that the two algorithms converge at comparable rates.  

Our convergence analysis of FWLP-P, presented in Section~\ref{sec:convergence}, first introduces a potential function $U_k$ which we show that for the iterates
of FWLP-P has distance $\bigO(1/\sqrt{k})$ from zero after $k$ iterations. To complete the analysis, we show that as $U_k \to 0$,
FWLP-P and FWLP converge to a primal-dual solution of \eqref{eq:std-lp}. The potential $U_k$ is similar to a traditional primal-dual optimality gap as discussed in Section~\ref{sec:traditionalgap}. A secondary contribution of this paper is the discussion in Section~\ref{sec:implementation} of how
FWLP and FWLP-P can be implemented efficiently in a way that only part of the matrix $A$ is needed at each iteration.

\section{FWLP and FWLP-P}
\label{sec:algorithms}
At each iteration, FWLP performs a Frank-Wolfe update on the primal and the dual of \eqref{eq:fwlp_saddle} using the step-size
$1/(k+1)$ instead of the standard step-size of $1/(k+2)$. The step-size $1/(k+1)$ was first analyzed by Freund and Grigas in \cite{Freund2016}.
Notably, FWLP uses the information obtained from the primal update, $\x_{k+1}$, in the computation of the dual update:
\begin{align}
  \r_{k+1} &:= \argmin_{\r} \left\{(\c-A^T\y_k)^T\r:
  \r \in \Delta\right\},
  \label{eq:fwlp-rk1}
  \\
  \x_{k+1}&:=\frac{k}{k+1}\x_k + \frac{1}{k+1}\r_{k+1},
  \label{eq:fwlp-xk1}
  \\
  \s_{k+1}&:=\argmax_{\s}\left\{(\b-A\x_{k+1})^T\s:
  \s\in\Gamma \right\},
  \label{eq:fwlp-sk1}
   \\
  \y_{k+1}&:=\frac{k}{k+1}\y_k+\frac{1}{k+1}\s_{k+1}.
  \label{eq:fwlp-yk1}
\end{align}
It is not hard to see that steps \eqref{eq:fwlp-rk1} and \eqref{eq:fwlp-sk1} above can be written in closed form (for more detail, see \cite[Chapter 6]{Hough2023}),
giving rise to Algorithm~\ref{alg:fwlp}. In this algorithm and for the remainder of the paper, $\e_i$ denotes the $i$th column of the identity matrix, whose length is determined from the context.

\begin{algorithm}[H]
\begin{algorithmic}[1]
\Require Starting points $\x_0\in\R^n_+, \y_0 \in\R^m$, constraint data $A$ and $\bm{b}$, and objective $\c$.
\vspace{0.2em}
\Statex \underline{Parameters:} $\xi, \eta > 0$ such that $2\lVert\x^*\rVert \leq \xi$ and $2\lVert\y^*\rVert \leq \eta$.
\For{$k=1,2,\ldots$}
    \State Determine $i = \argmin_{t} [\c - A^{\T}\y_k]_t$.
    \If{$[\c - A^{\T}\y_k]_{i} \geq 0$}
      \State Step towards zero in $\x$:
      \begin{align}\label{eqn:fwlp-x1}
        \x_{k+1} := \frac{k}{k+1}\x_k.
      \end{align}
    \Else
      \State Step toward $\xi$ for $x^{(i)}$, otherwise step toward zero for $x^{(j)}, j \neq i$:
      \begin{align}\label{eqn:fwlp-x2}
        \x_{k+1} &= \frac{k}{k+1}\x_k + \frac{\xi}{k+1}\e_i.
      \end{align}
    \EndIf \label{ln_main_end}
    \State Step towards $\pm \eta$ for each coordinate in $\y$ according to the sign pattern of $\b - A\x_{k+1}$:
    \begin{align}\label{eqn:fwlp-y1}
      \y_{k+1} &:= \frac{k}{k+1}\y_k + \frac{\eta}{k+1}\sgn(\b - A\x_{k+1}).
    \end{align}
\EndFor
\end{algorithmic}
\caption{FWLP: A primal-dual algorithm for \eqref{eq:std-lp} based on Frank-Wolfe \cite{Hough2023}.}
\label{alg:fwlp}
\end{algorithm}

FWLP-P adds the regularizing perturbations $\Vert \r\Vert^2/(2\sqrt{k})$ and $\Vert\s\Vert^2/(2\sqrt{k})$
to steps \eqref{eqn:fwlp-x1}/\eqref{eqn:fwlp-x2} and \eqref{eqn:fwlp-y1} in the above description of FWLP yielding \eqref{eq:rk1} and \eqref{eq:sk1} respectively in Algorithm~\ref{alg:fwlp-p}.
We analyze this algorithm in the next section.

Note that the computation of $\s_{k+1}$ in FWLP-P is a separable optimization problem allowing a simple median algorithm for each coordinate entry of $\s_{k+1}$.
The computation of $\r_{k+1}$ is more elaborate, but it requires only evaluation of the smallest
(most negative) elements of $\c - A^T\y_k$. We cover this in more detail in Section~\ref{sec:implementation} but of note here is that
FWLP-P retains some of the benefit of FWLP, namely sub-linear time computation per iteration, since only a part of $A$ is required. Note also that
the formulas \eqref{eq:rk1} and \eqref{eq:sk1} for $\r_{k+1}$ and $\s_{k+1}$ can be written equivalently as projections:
\begin{align}
  \r_{k+1}&:=\proj_{\Delta}(\sqrt{k}(A^T\y_k-\c)),
  \label{eq:rkproj}\\
  \s_{k+1}&:=\proj_{\Gamma}(\sqrt{k}(\b-A\x_{k+1})).
  \label{eq:skproj}
\end{align}
We discuss how to compute these projections efficiently in Section~\ref{sec:efficient_fwlp}.
\begin{algorithm}[H]
\begin{algorithmic}[1]
\Require Starting points $\x_0\in\R^n_+, \y_0 \in\R^m$, constraint data $A$ and $\bm{b}$, and objective $\c$.
\vspace{0.2em}
\Statex \underline{Parameters:} $\xi, \eta > 0$ such that $2\lVert\x^*\rVert \leq \xi$ and $2\lVert\y^*\rVert \leq \eta$.
\For{$k=1,2,\ldots$}
  \begin{align}
    \r_{k+1} &:= \argmin_{\r} \left\{(\c-A^T\y_k)^T\r + \frac{\Vert \r\Vert^2}{2\sqrt{k}}:
    \r\in\Delta\right\},
    \label{eq:rk1}
    \\
    \x_{k+1}&:=\frac{k}{k+1}\x_k + \frac{1}{k+1}\r_{k+1},
    \label{eq:xk1}
    \\
    \s_{k+1}&:=\argmax_{\s}\left\{(\b-A\x_{k+1})^T\s -\frac{\Vert\s\Vert^2}{2\sqrt{k}}:
    \s\in\Gamma\right\},
    \label{eq:sk1}
     \\
    \y_{k+1}&:=\frac{k}{k+1}\y_k+\frac{1}{k+1}\s_{k+1}.
    \label{eq:yk1}
  \end{align}
\EndFor
\end{algorithmic}
\caption{FWLP-P: FWLP with perturbations.}
\label{alg:fwlp-p}
\end{algorithm}

\section{Convergence analysis of FWLP-P} \label{sec:convergence}
The forthcoming analysis derives a recursion
\eqref{eq:recursion} for potential $U_k$ below satisfied by the iterates of FWLP-P. The recursion has
two perturbation terms $\delta_{k+1}$ and $\eps_{k+1}$ that are subsequently
bounded in \eqref{eq:epsk1_bound} and \eqref{eq:deltak1_bound}. We show that the quantity $U_k$ goes to 0 like $1/\sqrt{k}$.

\subsection{Deriving the recursion and potential function}
Let us introduce
\begin{equation}
  \delta_{k+1}:=\r_{k+2}^T(\c-A^T\y_{k+1})+\frac{1}{2\sqrt{k+1}}\Vert\r_{k+2}\Vert^2
  -\r_{k+1}^T(\c-A^T\y_{k+1})-\frac{k}{2(k+1)\sqrt{k}}\Vert \r_{k+1}\Vert^2.
  \label{eq:deltak1}
\end{equation}
Rearrange this equation:
\begin{equation}
  \delta_{k+1}-\r_{k+2}^T(\c-A^T\y_{k+1})-\frac{1}{2\sqrt{k+1}}\Vert\r_{k+2}\Vert^2
  =-\r_{k+1}^T(\c-A^T\y_{k+1})-\frac{k}{2(k+1)\sqrt{k}}\Vert \r_{k+1}\Vert^2.
  \label{eq:del1}
\end{equation}
Rewrite the first-term on the RHS of \eqref{eq:del1},
that is, $-\r_{k+1}^T(\c-A^T\y_{k+1})$,
using
\eqref{eq:yk1}:
\begin{align}
  -\r_{k+1}^T(\c-A^T\y_{k+1})&=-\r_{k+1}^T\left(\c-A^T\left(
  \frac{k}{k+1}\y_k+\frac{1}{k+1}\s_{k+1}\right)\right) \notag \\
  &=-\frac{k}{k+1}\r_{k+1}^T(\c-A^T\y_k)-\frac{1}{k+1}\r_{k+1}^T(\c-A^T\s_{k+1}).
  \label{eq:rewr1}
\end{align}
Substitute \eqref{eq:rewr1} for the first term on the RHS in \eqref{eq:del1},
moving the second term of \eqref{eq:rewr1}
$-\r_{k+1}^T(\c-A^T\s_{k+1})/(k+1)$ to the LHS, thereby rewriting \eqref{eq:del1} as:
\begin{multline}
\delta_{k+1}-\r_{k+2}^T(\c-A^T\y_{k+1})-\frac{1}{2\sqrt{k+1}}\Vert\r_{k+2}\Vert^2
+\frac{1}{k+1}\r_{k+1}^T(\c-A^T\s_{k+1}) \\
\shoveright{= -\frac{k}{k+1}\r_{k+1}^T(\c-A^T\y_k)-\frac{k}{2(k+1)\sqrt{k}}\Vert \r_{k+1}\Vert^2}
\\
=\frac{k}{k+1}\left(-\r_{k+1}^T(\c-A^T\y_k)-\frac{1}{2\sqrt{k}}\Vert \r_{k+1}\Vert^2\right).
  \label{eq:del2}
\end{multline}
The point of this
algebra is that the 2nd and 3rd terms
of the LHS correspond to the RHS advanced from $k$ to $k+1$,
thus setting up some of the terms of the recursion.

Next, introduce for $k \geq 2$
\begin{equation}
\eps_{k+1} := \frac{k}{k+1}\left(
-\s_{k+1}^T(\b-A\x_k)+\frac{\sqrt{k+1}}{2k}\Vert\s_{k+1}\Vert^2
+\s_k^T(\b-A\x_k)-\frac{1}{2\sqrt{k}}\Vert\s_k\Vert^2\right).
\label{eq:epsk1}
\end{equation}
Rewrite this equation as  
\begin{equation}
\eps_{k+1} + \frac{k}{k+1}\left(
\s_{k+1}^T(\b-A\x_k)-\frac{\sqrt{k+1}}{2k}\Vert\s_{k+1}\Vert^2\right)=
\frac{k}{k+1}\left(
\s_k^T(\b-A\x_k)-\frac{1}{2\sqrt{k}}\Vert\s_k\Vert^2\right).
\label{eq:eps1}
\end{equation}
Rearranging \eqref{eq:xk1} yields
$\x_{k}=(k+1)\x_{k+1}/k-\r_{k+1}/k$.  This means that the factor
$\s_{k+1}^T(\b-A\x_k)$ appearing in the LHS of \eqref{eq:eps1}
may be rewritten
\begin{align}
  \s_{k+1}^T(\b-A\x_k) &= \s_{k+1}^T\left(\b-A\left(\frac{k+1}{k}\x_{k+1}-\frac{1}{k}\r_{k+1}\right)\right) \notag \\
  &= \frac{k+1}{k}\s_{k+1}^T(\b-A\x_{k+1})-\frac{1}{k}\s_{k+1}^T(\b-A\r_{k+1}).
  \label{eq:rewr2}
\end{align}
Thus, the LHS of \eqref{eq:eps1} is rewritten:
\begin{align}
  \mbox{LHS of \eqref{eq:eps1}} &=
\eps_{k+1} + \frac{k}{k+1}\left(
\frac{k+1}{k}\s_{k+1}^T(\b-A\x_{k+1})-\frac{1}{k}\s_{k+1}^T(\b-A\r_{k+1})-\frac{\sqrt{k+1}}{2k}\Vert\s_{k+1}\Vert^2\right) \notag
\\
&=
\eps_{k+1} +
\s_{k+1}^T(\b-A\x_{k+1})
-\frac{1}{k+1}\s_{k+1}^T(\b-A\r_{k+1})
-\frac{1}{2\sqrt{k+1}}\Vert\s_{k+1}\Vert^2.
\label{eq:rewr3}
\end{align}

Thus, we have rewritten \eqref{eq:eps1} as
\begin{multline}
\eps_{k+1} +
\s_{k+1}^T(\b-A\x_{k+1})
-\frac{1}{k+1}\s_{k+1}^T(\b-A\r_{k+1})
-\frac{1}{2\sqrt{k+1}}\Vert\s_{k+1}\Vert^2
\\
=
\frac{k}{k+1}\left(
\s_k^T(\b-A\x_k)-\frac{1}{2\sqrt{k}}\Vert\s_k\Vert^2\right).
\label{eq:eps2}
\end{multline}
Here, we see the correspondence between the 2nd and 4th term on the LHS
with the two terms on the RHS (with $k$ advanced by 1).

Multiply \eqref{eq:xk1} by $\c$ and rearrange to obtain
\begin{equation}
\c^T\x_{k+1}-\frac{1}{k+1}\c^T\r_{k+1}=\frac{k}{k+1}\c^T\x_k.
\label{eq:cxk1}
\end{equation}
Similarly, from \eqref{eq:yk1},
\begin{equation}
-\b^T\y_{k+1}+\frac{1}{k+1}\b^T\s_{k+1}=-\frac{k}{k+1}\b^T\y_k.
\label{eq:byk1}
\end{equation}
Now add
\eqref{eq:del2},
\eqref{eq:eps2},
\eqref{eq:cxk1},
and \eqref{eq:byk1}, noting that many quantities on the LHS cancel,
while all quantities on the RHS contain the factor $k/(k+1)$, to
obtain a recursion:
\begin{equation}
\delta_{k+1}+\eps_{k+1} + U_{k+1} = \frac{k}{k+1} U_k,
\label{eq:recursion}
\end{equation}
where for $k\geq 2$
\begin{equation}
U_k:= -\r_{k+1}^T(\c-A^T\y_k)-\frac{1}{2\sqrt{k}}\Vert\r_{k+1}\Vert^2
+\s_k^T(\b-A\x_k) - \frac{1}{2\sqrt{k}}\Vert\s_k\Vert^2 + \c^T\x_k-\b^T\y_k.
\label{eq:udef}
\end{equation}
We call $U_k$ the potential function, and we will use this in our convergence analysis to bound the distance from optimality of \eqref{eq:std-lp} and \eqref{eq:std-dual}.

\subsection{Bounding the potential function}
Our next task is to lower bound $\delta_{k+1}$ and $\eps_{k+1}$ so that we
can use \eqref{eq:recursion} to develop a more useful bound on the potential function.

\begin{lemma} \label{lem:epsk1_bound}
  Recall $\epsilon_{k+1}$ defined in \eqref{eq:epsk1}. We have the bound
  \[
    \eps_{k+1}\ge -\frac{m\eta^2}{6k^2\sqrt{k-1}}.
  \]
\end{lemma}
\begin{proof}
  By adding and subtracting
  multiples of 
  $\Vert \s_{k+1}\Vert^2$ and
  and $\Vert \s_{k}\Vert^2$ inside \eqref{eq:epsk1}, we obtain
  \[
  \eps_{k+1}=\frac{k}{k+1}(\eps_{k+1}'+\eps_{k+1}''+\eps_{k+1}'''),
  \]
  where
  \begin{align*}
  \eps_{k+1}'&:=
  -\s_{k+1}^T(\b-A\x_k)+\frac{1}{2\sqrt{k-1}}\Vert\s_{k+1}\Vert^2
  +\s_k^T(\b-A\x_k)-\frac{1}{2\sqrt{k-1}}\Vert\s_k\Vert^2, \\
  \eps_{k+1}'' &:= \left(\frac{1}{2\sqrt{k-1}}-\frac{1}{2\sqrt{k}}\right)\Vert \s_k\Vert^2, \\
  \eps_{k+1}''' & :=
  \left(\frac{\sqrt{k+1}}{2k}-\frac{1}{2\sqrt{k-1}}\right)\Vert\s_{k+1}\Vert^2.
  \end{align*}
  We see that $\eps_{k+1}'\ge 0$ because $\s_k$ is the maximizer
  of $\s^T(\b-A\x_k)-\Vert\s\Vert^2/(2\sqrt{k-1})$ according to the
  definition \eqref{eq:sk1}, while $\s_{k+1}$ is some other feasible point,
  and $\eps_{k+1}'$ is the difference between the two objective values.
  We also see that $\eps_{k+1}''\ge 0$.

  As for $\eps_{k+1}'''$, use the estimate
  $\Vert\s_k\Vert\le \sqrt{m}\eta$ according to \eqref{eq:sk1}.
  Finally, by taking a common
  denominator and observing that $(k-1/(3k))^2\le k^2-1$,
  we obtain a lower bound of $-1/(6k^2\sqrt{k-1})$ on the parenthesized factor.
  Thus, adding the three contributions,
  \begin{equation} \label{eq:epsk1_bound}
    \eps_{k+1}\ge -\frac{m\eta^2}{6k^2\sqrt{k-1}} \cdot \frac{k}{k+1} \geq -\frac{m\eta^2}{6k^2\sqrt{k-1}}.
  \end{equation}
\end{proof}

\begin{lemma} \label{lem:deltak1_bound}
  Recall $\delta_{k+1}$ defined in \eqref{eq:deltak1}. We have the bound
  \[
    \delta_{k+1}\ge -D/k^{3/2},
  \]
  where $D$ is a positive constant depending on the data defined in \eqref{eq:Ddefn} below.
\end{lemma}
\begin{proof}
  Split $\delta_{k+1}$ into four terms:
  $$\delta_{k+1}=\delta_{k+1}' + \delta_{k+1}'' + \delta_{k+1}''' + \delta_{k+1}^{\mathrm{iv}}$$
  where
  \begin{align*}
    \delta_{k+1}'  &:=
  \r_{k+2}^T(\c-A^T\y_{k})+\frac{1}{2\sqrt{k}}\Vert\r_{k+2}\Vert^2
    -\r_{k+1}^T(\c-A^T\y_{k})-\frac{1}{2\sqrt{k}}\Vert \r_{k+1}\Vert^2,
    \\
    \delta_{k+1}''&:= (\r_{k+2}-\r_{k+1})^TA^T(\y_k-\y_{k+1}),
    \\
    \delta_{k+1}'''&:=\left(\frac{1}{2\sqrt{k+1}}-\frac{1}{2\sqrt{k}}\right)
    \Vert \r_{k+2}\Vert^2,
    \\
    \delta_{k+1}^{\mathrm{iv}} &:=\left(\frac{1}{2\sqrt{k}} - \frac{k}{2(k+1)\sqrt{k}}\right)
    \Vert \r_{k+1}\Vert^2.
  \end{align*}

  We observe that $\delta_{k+1}'\ge 0$ because $\r_{k+1}$ is the
  minimizer of $\r^T(\c-A^T\y_k)+\Vert\r\Vert^2/(2\sqrt{k})$ according to
  \eqref{eq:rk1}, whereas $\r_{k+2}$ is some other feasible point.

  Next, we turn to $\delta_{k+1}''$, a product of three factors.
  Starting on the first factor,
  \begin{align}
  \Vert \r_{k+2}-\r_{k+1}\Vert &= 
  \Vert \proj_{\Delta}(\sqrt{k+1}(A^T\y_{k+1}-\c)) - 
  \proj_{\Delta}(\sqrt{k}(A^T\y_k-\c)\Vert
  \notag \\
  &\le
  \Vert \sqrt{k+1}(A^T\y_{k+1}-\c)) - 
  \sqrt{k}(A^T\y_k-\c)\Vert \label{eq:lipschitz}\\
  &=
  \Vert (\sqrt{k+1}-\sqrt{k})(A^T\y_{k+1}-\c)+
  \sqrt{k}(A^T\y_{k+1}-\c-(A^T\y_k-\c))\Vert \notag \\
  &=
  \Vert (\sqrt{k+1}-\sqrt{k})(A^T\y_{k+1}-\c)+
  \sqrt{k}A^T(\y_{k+1}-\y_k)\Vert \notag \\
  &\le
  |\sqrt{k+1}-\sqrt{k}|\cdot \Vert A^T\y_{k+1}-\c\Vert +
  \sqrt{k}\Vert A\Vert\cdot \Vert \y_{k+1}-\y_k\Vert. \notag
  \end{align}
  Here, the first line follows from \eqref{eq:rkproj},
  the second (that is, \eqref{eq:lipschitz})
  from the fact that the Lipschitz constant of
  $\proj_C(\cdot)$ is 1 for any closed nonempty convex set $C$,
  the third line adds and subtracts the same term, and the
  last line applies the triangle inequality and
  submultiplicativity.

  Now we use the facts
  that $\sqrt{k+1}-\sqrt{k}\le 1/\sqrt{k}$,
  and, from \eqref{eq:yk1},
  \[
  \y_{k+1}-\y_{k}=\frac{1}{k+1}(\s_{k+1}-\y_k),
  \]
  and finally, the bounds $\Vert\y_k\Vert\le \sqrt{m}\eta$,
  $\Vert\s_k\Vert \le \sqrt{m}\eta$ 
  to conclude that 
   $\Vert \y_{k+1}-\y_{k}\Vert\le 2\sqrt{m}\eta/(k+1)$ and thus
  \[
  \Vert \r_{k+2}-\r_{k+1}\Vert \le (1/\sqrt{k})\cdot (\Vert A\Vert\sqrt{m}\eta+\Vert\c\Vert)
  +(2/\sqrt{k})\Vert A\Vert \sqrt{m}\eta.
  \]

  This takes care of the first factor in $\delta_{k+1}''$.  The middle
  factor is bound by $\Vert A\Vert$, and the third factor $\Vert \y_{k+1}-\y_{k}\Vert$
  is bounded by
  (see the previous paragraph) $2\sqrt{m}\eta/(k+1)$.  Thus, overall,
  we obtain
  \begin{align*}
    \delta_{k+1}'' &\ge -\Vert \r_{k+2}-\r_{k+1}\Vert\cdot \Vert A\Vert
    \cdot \Vert\y_k-\y_{k+1}\Vert \\
    &\ge -C/k^{3/2},
  \end{align*}
  where $C$ depends on the problem data:
  \[
  C\le 2\Vert A\Vert \sqrt{m}\eta(3\Vert A \Vert \sqrt{m}\eta + \Vert\c\Vert).
  %4\Vert A\Vert\cdot (\max(\Vert A\Vert,1)+\Vert\c\Vert)m\eta^2.
  \]

  For $\delta_{k+1}'''$, by finding a common denominator and then multiplying the
  resulting fraction by $\sqrt{k}+\sqrt{k+1}$, we obtain
  \[
  \delta_{k+1}'''\ge -1/(4k^{3/2})\cdot \Vert \r_{k+2}\Vert^2
  \ge -1/(4k^{3/2})\cdot \xi^2.
  \]
  Finally, one sees that $\delta_{k+1}^{\mathrm{iv}}\ge 0$.  Putting all
  of these terms together yields:
  \begin{equation} \label{eq:deltak1_bound}
    \delta_{k+1}\ge -D/k^{3/2},
  \end{equation}
  where
  \begin{equation} \label{eq:Ddefn}
    D :=  2\Vert A\Vert \sqrt{m}\eta(3\Vert A \Vert \sqrt{m}\eta + \Vert\c\Vert)+\frac{\xi^2}{4}.
    %4\Vert A\Vert\cdot (\max(\Vert A\Vert,1)+\Vert\c\Vert)m\eta^2 + \frac{\xi^2}{4}.
  \end{equation}
\end{proof}

%\begin{remark}
  Combining \eqref{eq:recursion}, Lemma~\ref{lem:epsk1_bound}, and Lemma~\ref{lem:deltak1_bound}, we now have for all $k \geq 2$
  \[
    U_{k+1} \leq \frac{k}{k+1}U_k + \frac{D}{k^{3/2}} + \frac{m\eta^2}{6k^2\sqrt{k-1}}.
  \]
  In fact, since $\sqrt{k} \geq 1/\sqrt{k-1}$ for all $k \geq 2$, we can enlarge $D$ to obtain the bound
  \begin{equation} \label{eq:recursion2}
    U_{k+1} \leq \frac{k}{k+1}U_k + \frac{\bar{D}}{k^{3/2}},
  \end{equation}
  where
  \[
    \bar{D} := D + \frac{m\eta^2}{6}.
  \]
  With this bound, we prove the following bound on $U_k$.
%\end{remark}

\begin{theorem} \label{thm:Ukdec}
  Recall $U_k$ defined in \eqref{eq:udef}. We have the bound
  \[
    U_{k+1} \leq \frac{F}{\sqrt{k}},
  \]
  for all $k\geq 2$, where $F := \max\{\sqrt{2}U_2, 6\bar{D}\}$.
\end{theorem}
\begin{proof}  The $k=2$ case follows immediately from the definition of $F$.
Suppose inductively
that $U_k\le F/\sqrt{k}$.  Then we check:
\begin{align*}
  U_{k+1}&\le \frac{k}{k+1}U_k + \bar{D}/k^{3/2} \\
  &\le \frac{kF}{(k+1)\sqrt{k}} + \bar{D}/k^{3/2} \\
  &=\frac{k^2F+\bar{D}(k+1)}{(k+1)k^{3/2}} \\
  &=\frac{1}{\sqrt{k+1}}\cdot\frac{k^2F+\bar{D}(k+1)}{\sqrt{k+1}\cdot k^{3/2}} \\
  &\le\frac{1}{\sqrt{k+1}}\cdot\frac{k^2F+\bar{D}(k+1)}{(\sqrt{k}+1/(3\sqrt{k}))k^{3/2}} \\
  &=\frac{1}{\sqrt{k+1}}\cdot\frac{k^2F+\bar{D}(k+1)}{k^2+k/3} \\
  &\le\frac{1}{\sqrt{k+1}}\cdot\frac{k^2F+kF/3}{k^2+k/3} \\
  &=\frac{1}{\sqrt{k+1}}\cdot F.
\end{align*}
Here, we used \eqref{eq:recursion2} for the
first line, the induction hypothesis for the
second line, the inequality
$\sqrt{k+1}\ge\sqrt{k}+1/(3\sqrt{k})$ for $k\ge 1$ on the 5th line (easy
to confirm by squaring)
and the assumption $\bar{D}\le F/6$ on the 7th line, which implies for all $k\ge 1$
that $\bar{D}(k+1)\le Fk/3$.

Thus, $U_k\le F/\sqrt{k}$.
\end{proof}

The perturbation terms from \eqref{eq:rk1} and \eqref{eq:sk1} in FWLP-P
were essential to this analysis. Without them, as in FWLP, there is no useful bound
on how much the primal step $\r_{k+2}$ can differ from $\r_{k+1}$ when the index of the most-violated
constraint changes from one iteration to the next. With the perturbations, we are able to obtain the inequality
\eqref{eq:lipschitz}.

\subsection{Proving convergence of the potential function implies convergence of the iterates}
We have shown that the potential function $U_k$ decreases at a rate of $\bigO(1/\sqrt{k})$ as $k$ increases.
It remains to show that the potential function bounds the distance of the iterates from an optimizer.
We first need the following lemmas.  Both lemmas have the same flavor: if an infeasible point for an LP improves on the optimal objective value, then the amount of objective improvement is bounded in terms of the amount of infeasibility.

\begin{lemma} \label{lem:primal-bound}
  Let $p^*$ denote the optimal value of \eqref{eq:std-lp}.
  Assume $p^*$ is finite, i.e., \eqref{eq:std-lp} is feasible and bounded.
  Let $(\x^*,\y^*)$ be an arbitrary optimizing primal-dual pair of \eqref{eq:std-lp} and \eqref{eq:std-dual}.
  Then for an arbitrary $\hat\x\ge\bz$,
  \begin{equation}
  \c^T\hat\x\ge p^*-\Vert\y^*\Vert_\infty\Vert\b-A\hat\x\Vert_1.
  \label{eq:primst}
  \end{equation}
  \label{lem:primdist}
\end{lemma}

\begin{proof}
  Select an arbitrary optimizing pair $(\x^*,\y^*)$ for \eqref{eq:std-lp}.
  Let $\k:=\b-A\hat\x$.  Consider the LP,
  \begin{equation}
  \begin{array}{rl}
    \min_{\x} & \c^T\x \\
    \mbox{s.t.} & A\x=\b-\k,\\
    & \x\ge \bz.
  \end{array}
  \label{eq:sklp}
  \end{equation}
  This LP is clearly feasible since $\hat\x$ satisfies the constraints.
  It is also
  bounded, as argued by contradiction:
  If   \eqref{eq:sklp} were feasible and unbounded, then there would exist
  a certificate of unboundedness, which is a vector
  $\w$ such that $\w\ge\bz$, $A\w=\bz$, and $\c^T\w<0$.  However,
  such a $\w$ would also certify unboundedness of
  \eqref{eq:std-lp}, which we have already assumed to be bounded.

  Therefore, the dual of \eqref{eq:sklp}, which is,
  \begin{equation}
  \begin{array}{rl}
    \max_{\y} & (\b-\k)^T\y \\
    \mbox{s.t.} & A^T\y\le \c,
  \end{array}
  \label{eq:skdlp}
  \end{equation}
  has an optimal solution, say $\hat\y$.  Since $\y^*$ is also feasible
  for \eqref{eq:skdlp}, we have the following chain of inequalities
  \begin{align*}
    p^*&=\b^T\y^* && \mbox{(by strong duality of \eqref{eq:std-lp})} \\
        &=(\b-\k)^T\y^* +\k^T\y^* \\
        &\le (\b-\k)^T\y^* +\Vert \k\Vert_1\cdot \Vert \y^*\Vert_\infty \\
        &\le (\b-\k)^T\hat\y +\Vert \k\Vert_1\cdot \Vert \y^*\Vert_\infty
        &&\mbox{(since $\hat\y$ maximizes \eqref{eq:skdlp})} \\
        &\le \c^T\hat\x + \Vert\k\Vert_1\cdot\Vert\y^*\Vert_\infty.
        &&\mbox{(by weak duality between \eqref{eq:sklp} and \eqref{eq:skdlp})}
  \end{align*}
  Recalling $\k=\b-A\hat\x$, the final line in this chain
  establishes \eqref{eq:primst}.
\end{proof}

\begin{lemma} \label{lem:dual-bound}
  Let $d^*$ denote the optimal value of \eqref{eq:std-dual}.
  Assume $d^*$ is finite, i.e., \eqref{eq:std-dual} is feasible and bounded.
  Let $(\x^*,\y^*)$ be an arbitrary optimizing primal-dual pair of \eqref{eq:std-lp} and \eqref{eq:std-dual}.
  Then for an arbitrary $\hat\y$,
  \begin{equation}
  \b^T\hat\y\le d^*+\Vert\x^*\Vert_1\cdot l,
  \label{eq:dumst}
  \end{equation}
  where $l$ measures the infeasibility of $\hat\y$, that is,
  \begin{equation}
  l:=\max(0,\max_j\{\e_j^T(A^T\hat\y-\c)\}).
  \label{eq:ldef}
  \end{equation}
  \label{lem:dualdist}
\end{lemma}

\begin{proof}
  This proof is analogous to the previous proof.  Let
  $(\x^*,\y^*)$ be an arbitrary primal-dual optimizer
  of \eqref{eq:std-dual}.  Consider the dual-form LP given by
  \begin{equation}
  \begin{array}{rl}
    \max_{\y} & \b^T\y \\
    \mbox{s.t.} & A^T\y\le \c+l\e,
  \end{array}
  \label{eq:dldu}
  \end{equation}
  where $l$ is from \eqref{eq:ldef} and $\e$ denotes
  the vector of all 1's.  This LP is clearly feasible since
  $\hat\y$ satisfies the constraint.  Furthermore, it is
  bounded as argued by contradiction.  If
  \eqref{eq:dldu} were feasible and unbounded, there would
  exist a certificate of unboundedness, that is, a vector
  $\z$ such that $A^T\z\le\bz$ and $\b^T\z>0$.  However, this
  certificate would also certify the unboundedness of
  \eqref{eq:std-dual}, which is assumed to be bounded.

  Therefore, the dual of \eqref{eq:dldu}, which is
  \begin{equation}
  \begin{array}{rl}
    \min_{\x} & (\c+l\e)^T\x \\
    \mbox{s.t.} & A\x=\b, \\
    & \x\ge\bz,
  \end{array}
  \label{eq:pldu}
  \end{equation}
  has an optimal solution which we denote $\hat\x$.
  Since $\x^*$ is also feasible for \eqref{eq:pldu},
  we have the following chain of inequalities:
  \begin{align*}
    d^* &= \c^T\x^* &&\mbox{(by strong duality of \eqref{eq:std-dual})} \\
    &=(\c+l\e)^T\x^* -     l\e^T\x^* \\
    &\ge (\c+l\e)^T\x^* - l\Vert\x^*\Vert_1 \\
    &\ge (\c+l\e)^T\hat\x - l\Vert\x^*\Vert_1 &&
    \mbox{(by the optimality of $\hat\x$ in \eqref{eq:pldu})} \\
    &\ge \b^T\hat\y - l\Vert\x^*\Vert_1 &&
    \mbox{(by weak duality betwen \eqref{eq:dldu} and \eqref{eq:pldu})}.
  \end{align*}
  The final line after rearrangement is \eqref{eq:dumst}.
\end{proof}

We are now ready to prove the penultimate theorem that shows convergence of FWLP-P. We show that, assuming a primal-dual optimal LP solution exists and $\xi,\eta$ have been chosen correctly, the $U_k$ plus terms that tend to 0 bound the distance from optimality. %if $U_k$ is bounded away from zero,
%then the distance from optimality of the iterates is governed by how far $U_k$ is from zero plus terms that go to zero as $k$ tends
%to infinity.  
Note that an analog of Theorem~\ref{thm:boundedUk_optdist} can be derived for FWLP using a similar proof. This analog is not presented here since we are not able to show that $U_k\to 0$ for FWLP.

\begin{theorem} \label{thm:boundedUk_optdist}
  Suppose 
  %$U_k \leq \gamma$ and 
  $k\geq 2$. Then,
  \begin{align}
    &\x_k \ge \bz, \label{eq:op1} \\
    &\Vert\b-A\x_k\Vert_1 \le \frac{2U_k}{\eta} + \frac{\xi^2}{\eta\sqrt{k}} + \frac{\eta}{\sqrt{k-1}}, \label{eq:op2}\\
    &\max(0,\max_j\{\e_j^T(A^T\y_k-\c)\}) \le \frac{2U_k}{\xi} + \frac{\xi}{\sqrt{k}} + \frac{\eta^2}{\xi\sqrt{k-1}}, \label{eq:op3}\\
    &\c^T\x_k-\b^T\y_k\le U_k,\label{eq:op4}
    \end{align}
  provided
  \begin{equation}
  \xi \ge 2\Vert\x^*\Vert_1, \qquad \eta\ge 2\Vert\y^*\Vert_\infty.
  \label{eq:xietaassum}
  \end{equation}
\label{thm:polyopt}
\end{theorem}
\begin{proof}
  %First, observe that all iterates satisfy $\x_k \geq \bz$. Now recall
  %\begin{equation}
  %  U_k = -\r_{k+1}^T(\c-A^T\y_k)-\frac{1}{2\sqrt{k}}\Vert\r_{k+1}\Vert^2
  %  +\s_k^T(\b-A\x_k) - \frac{1}{2\sqrt{k}}\Vert\s_k\Vert^2 + %\c^T\x_k-\b^T\y_k.
  %\end{equation}
  %Suppose 
  %$U_k \leq \gamma$ and 
  %the assumptions \eqref{eq:xietaassum} hold, and let
  %$\gamma:=U_k$.
  It is immediate from the algorithm that \eqref{eq:op1} holds. 
 Next, notice that $\s_k$ is a solution to
  \begin{equation*}
    \max_{\s \in \Gamma} \left\{ \s^T(\b-A\x_k) - \frac{1}{2\sqrt{k-1}}\Vert\s\Vert^2\right\},
  \end{equation*}
  so for arbitrary $\s \in \Gamma$,
  \begin{equation} \label{eq:sk_compare}
    \begin{split} 
    \s_k^T(\b-A\x_k) - \frac{1}{2\sqrt{k}}\Vert\s_k\Vert^2 &\geq \s_k^T(\b-A\x_k) - \frac{1}{2\sqrt{k-1}}\Vert\s_k\Vert^2,\\
                                                           &\geq \s^T(\b-A\x_k) - \frac{1}{2\sqrt{k-1}}\Vert\s\Vert^2.
    \end{split}
  \end{equation}
  Setting $\s = \bz$ gives
  \[
    \s_k^T(\b-A\x_k) - \frac{1}{2\sqrt{k}}\Vert\s_k\Vert^2 \geq 0.
  \]
  Similarly, $\r_{k+1}$ is a solution to
  \begin{equation} \label{eqn:r-opt}
    \min_{\r} \left\{\r^T(\c-A^T\y_k)+\frac{1}{2\sqrt{k}}\Vert\r\Vert^2\right\},
  \end{equation}
  so for arbitrary $\r \in \Delta$,
  \begin{equation} \label{eq:rk_compare}
    \r_{k+1}^T(\c-A^T\y_k)+\frac{1}{2\sqrt{k}}\Vert\r_{k+1}\Vert^2 \leq \r^T(\c-A^T\y_k)+\frac{1}{2\sqrt{k}}\Vert\r\Vert^2.
  \end{equation}
  Setting $\r = \bz$ gives
  \begin{equation} \label{eq:rk_compare_zero}
    -\r_{k+1}^T(\c-A^T\y_k)-\frac{1}{2\sqrt{k}}\Vert\r_{k+1}\Vert^2 \geq 0.
  \end{equation}
  The above two results imply with \eqref{eq:udef} that 
  %$\gamma \geq 
  $U_k \geq \c^T\x_k - \b^T\y_k$, thus establishing \eqref{eq:op4}.

  Setting $\s := \eta\cdot\sgn(\b - A\x_k) \in \Gamma$ in \eqref{eq:sk_compare} gives the bound
  \begin{equation} \label{eq:sk_compare2}
    \s_k^T(\b - A\x_k) - \frac{1}{2\sqrt{k}}\lVert\s_k\rVert^2 \geq \eta\lVert \b - A\x_k\rVert_1 - \frac{1}{2\sqrt{k-1}}m\eta^2.
  \end{equation}
  Similarly, set $\r := \xi\e_j$ where $j = \argmax_j \{ \e_j^T(A^T\y_k - \c)\}$. Noting $\r \in \Delta$, we can use \eqref{eq:rk_compare}
  to obtain the bound
  \begin{equation*} 
      -\r_{k+1}^T(\c-A^T\y_k)-\frac{1}{2\sqrt{k}}\Vert\r_{k+1}\Vert^2 \geq \xi\cdot \max_j\{\e_j^T(A^T\y_k - \c)\} - \frac{1}{2\sqrt{k}}\xi^2.
  \end{equation*}
  In fact, from \eqref{eq:rk_compare_zero} the left-hand side above is nonnegative. It follows that we can strengthen the above bound to
  \begin{equation} \label{eq:rk_compare2}
      -\r_{k+1}^T(\c-A^T\y_k)-\frac{1}{2\sqrt{k}}\Vert\r_{k+1}\Vert^2 \geq \xi\cdot l - \frac{1}{2\sqrt{k}}\xi^2,
  \end{equation}
  where $l = \max(0,\max_j\{\e_j^T(A^T\y_k - \c)\})$.
  Using Lemma~\ref{lem:primal-bound}, Lemma~\ref{lem:dual-bound}, \eqref{eq:udef}, \eqref{eq:sk_compare2}, and \eqref{eq:rk_compare2}, we may write
  \begin{align*}
    %\gamma 
    U_k &= -\r_{k+1}^T(\c-A^T\y_k)-\frac{1}{2\sqrt{k}}\Vert\r_{k+1}\Vert^2
    + \s_k^T(\b-A\x_k) - \frac{1}{2\sqrt{k}}\Vert\s_k\Vert^2 + \c^T\x_k-\b^T\y_k,\\
           &\geq \xi\cdot l - \frac{1}{2\sqrt{k}}\xi^2
    + \eta\lVert \b - A\x_k\rVert_1 - \frac{1}{2\sqrt{k-1}}m\eta^2 + \c^T\x_k - \b^T\y_k,\\
           &\geq \xi\cdot l - \frac{1}{2\sqrt{k}}\xi^2
    + \eta\lVert \b - A\x_k\rVert_1 - \frac{1}{2\sqrt{k-1}}m\eta^2 + p^* - \lVert\y^*\rVert_{\infty}\lVert\b - A\x_k\rVert_1 - d^* - \lVert\x^*\rVert_1\cdot l,\\
           &\geq \xi\cdot l - \frac{1}{2\sqrt{k}}\xi^2
    + \eta\lVert \b - A\x_k\rVert_1 - \frac{1}{2\sqrt{k-1}}m\eta^2 + p^* - d^* - \frac{\eta}{2}\lVert\b - A\x_k\rVert_1 - \frac{\xi}{2}\cdot l\\
           &= \frac{\xi}{2}\cdot l - \frac{1}{2\sqrt{k}}\xi^2 + \frac{\eta}{2}\lVert \b - A\x_k\rVert_1 - \frac{1}{2\sqrt{k-1}}m\eta^2,
  \end{align*}
  where the fourth line used the assumption \eqref{eq:xietaassum} and the final line used the fact that $p^* = d^*$. We are left with the bound
  \begin{equation*}
    \frac{\xi}{2}\cdot l + \frac{\eta}{2}\lVert\b - A\x_k\rVert_1 \leq U_k + \frac{\xi^2}{2\sqrt{k}} + \frac{m\eta^2}{2\sqrt{k-1}},
  \end{equation*}
  where both terms on the left are nonnegative. It follows that each term must be individually bounded by the right-hand side above.
  This establishes \eqref{eq:op2} and \eqref{eq:op3}.
%  \begin{align*}
%    l &\leq \frac{2\gamma}{\xi} + \frac{\xi}{\sqrt{k}} + %\frac{m\eta^2}{\xi\sqrt{k-1}},\\
%    \lVert\b - A\x_k\rVert_1 &\leq \frac{2\gamma}{\eta} + %\frac{\xi^2}{\eta\sqrt{k}} + \frac{m\eta}{\sqrt{k-1}}.
%  \end{align*}
\end{proof}

\begin{theorem}
  The iterates of FWLP-P converge to an $\epsilon$-optimal solution of \eqref{eq:std-lp} and its dual \eqref{eq:std-dual} after $\bigO(1/\epsilon^2)$ iterations.
\end{theorem}
\begin{proof}
  This follows immediately by applying the bound from Theorem~\ref{thm:Ukdec} in Theorem~\ref{thm:boundedUk_optdist}.
\end{proof}

\section{Relating $\bm{U_k}$ to the standard primal-dual gap}
\label{sec:traditionalgap}
Consider the general saddle-point problem
\[
  \min_{\x\in\X}\max_{\y\in\Y} f(\x,\y).
\]
A standard measure of the optimality gap of an iteration $(\x_k,\y_k)$ used in the analysis of many primal-dual algorithms for saddle-point problems (for example, \cite{Chambolle2011, Gidel2017, Nemirovski2004}) is the primal-dual gap:
\begin{equation} \label{eq:pd_gap}
  \max \left\{f(\x_k,\hat{\y}_k) - f(\hat{\x}_k,\y_k) : \hat\x_k \in \X, \hat\y_k \in \Y \right\}.
\end{equation}
Obviously, the solution to \eqref{eq:pd_gap} is
\begin{align*}
  \hat{\x}_k = \argmin \{f(\x, \y_k) : \x \in \X\},\\
  \hat{\y}_k = \argmax \{f(\x_k, \y) : \y \in \Y\}.
\end{align*}
It is noted in \cite{PDLP2} that for the saddle-point problem associated with LP
\[
  \min_{\x\geq\bz}\max_{\y\in\R^m} \c^T\x + \y^T(\b - A\x),
\]
the primal-dual gap
can be infinite, since the feasible set $\R^n_+\times\R^m$ is unbounded. This is not an issue for the modified saddle-point formulation \eqref{eq:fwlp_saddle}
used in our analysis, since the feasible set $\Delta \times \Gamma$ is bounded by virtue of the redundant constraints. 

Recall the definition of $\L(\x,\y)$ from \eqref{eq:fwlp_saddle}.    Let 
\[\mathcal{M}_k:=\max\{\L(\x_k,\s)-\L(\r,\y_k):(\r,\s)\in\Delta\times \Gamma\},\]
which is the specialization of \eqref{eq:pd_gap}
to our setting.  We argue that $\mathcal{M}_k$ is a perturbation of $U_k$ via the following bound:
\begin{align*}
|\mathcal{M}_k-U_k|
&=
\left|
\max_{\s\in\Gamma}[(\b-A\x_k)^T\s+\c^T\x_k]-\min_{\r\in\Delta}[(\c-A^T\y_k)^T\r+\b^T\y_k]-U_k
\right|
\\
&\le
\left|
\max_{\s\in\Gamma}\left[(\b-A\x_k)^T\s+\c^T\x_k-\frac{\Vert\s\Vert^2}{2\sqrt{k-1}}\right]-\min_{\r\in\Delta}\left[(\c-A^T\y_k)^T\r+\b^T\y_k+\frac{\Vert\r\Vert^2}{2\sqrt{k}}\right]-U_k
\right| 
\\
&\hphantom{\le}\quad\mbox{}+
\max_{(\r,\s)\in\Delta\times\Gamma}
\left|
\frac{\Vert\s\Vert^2}{2\sqrt{k-1}}+\frac{\Vert\r\Vert^2}{2\sqrt{k}}
\right|
\\
&\le
\left|
\left[(\b-A\x_k)^T\s_k+\c^T\x_k-\frac{\Vert\s_k\Vert^2}{2\sqrt{k-1}}\right]-\left[(\c-A^T\y_k)^T\r_{k+1}+\b^T\y_k+\frac{\Vert\r_{k+1}\Vert^2}{2\sqrt{k}}\right]-U_k
\right|\\
&\hphantom{\le}\quad\mbox{}
+\frac{m\eta^2}{2\sqrt{k-1}}+\frac{\xi^2}{2\sqrt{k}}
\\
&\le
\frac{m\eta^2+\xi^2}{2\sqrt{k-1}} + \bigO(k^{-3/2}).
\end{align*}
Here, the second line adds and subtracts the same terms and then applies the triangle inequality for $|\cdot|$.  The third line uses the definitions of $\s_k$ from \eqref{eq:sk1} and $\r_k$ from \eqref{eq:rk1}.  The fourth line uses \eqref{eq:udef}, noting that the terms all cancel out except for the difference between a denominator of $2\sqrt{k-1}$ versus a denominator of $2\sqrt{k}$.  This small remainder is written as $\bigO(k^{-3/2})$ on the fourth line.

% Recall the definition of $\L(\x,\y)$ from \eqref{eq:fwlp_saddle}. Rearranging $U_k$ shows the relation between the primal-dual gap \eqref{eq:pd_gap} and
% the potential function $U_k$ we use in our analysis:
% \begin{align*}
%   U_k &= -\r_{k+1}^T(\c-A^T\y_k)-\frac{1}{2\sqrt{k}}\Vert\r_{k+1}\Vert^2
%   +\s_k^T(\b-A\x_k) - \frac{1}{2\sqrt{k}}\Vert\s_k\Vert^2 + \c^T\x_k-\b^T\y_k\\
%       &= \c^T\x_k + \s_k^T(\b - A\x_k) - \frac{1}{2\sqrt{k}}\lVert\s_k\rVert^2 - \c^T\r_{k+1} - \y_k^T(\b - A\r_{k+1}) - \frac{1}{2\sqrt{k}}\lVert\r_{k+1}\rVert^2\\
%       &= \max \left\{\L(\x_k, \s) - \L(\x_k, \r) : \s \in \Gamma, \r \in \Delta \right\} - \frac{1}{2\sqrt{k}}\left(\lVert\s_k\rVert^2 + \lVert\r_{k+1}\rVert^2\right),
% \end{align*}
% where the last line applies \eqref{eq:rk1} and \eqref{eq:sk1}.

\section{Efficient implementation of FWLP and FWLP-P} \label{sec:implementation}
A major advantage of FWLP and FWLP-P is their low computational cost per iteration. Naïve implementations of Algorithms~\ref{alg:fwlp} and \ref{alg:fwlp-p}
have iteration cost bounded by the cost of a full matrix-vector product. We discuss below how this can be significantly improved with an efficient implementation.
\subsection{FWLP}
Consider, e.g., the iteration $k=1$.  Suppose we store $A\x_1$. % and $A\e_i$ for each $i \in [n]$. 
In Algorithm~\ref{alg:fwlp} we can perform an extra step to update
%our knowledge of 
$A\x_2$ by either computing
\[
  A\x_2 = \frac{k}{2}A\x_1,
\]
or
\[
  A\x_2 = \frac{k}{2}A\x_1 + \frac{\xi}{2}A\e_i,
\]
where $i$ is the index of the most violated dual constraint, computed as $\argmin \{c_i - \e_i^TA^T\y_1\}$. Ignoring the cost of computing $i$,
such an update runs in $\bigO(m + n)$ operations per iteration (cost of updating $\x_k$ and $A\x_k$). Note that $\bigO(m+n)$ can be reduced to $\bigO(m)$ if we keep track of the product of scaling factors $k/(k+1)$ of $\x_{k+1}$ in a separate variable.
In contrast, the naïve implementation of the primal update in FWLP, where one instead performs matrix-vector products,
runs in $\bigO(mn)$ operations per iteration.

The dual update for, e.g., $k=1$, is given by
\[
  \y_2 = \frac{1}{2}\y_1 + \frac{\eta}{2}\sgn(\b - A\x_2),
\]
which also runs in $\bigO(m)$ since $A\x_2$ has already been computed in the primal step. The naïve implementation again takes $\bigO(mn)$ operations.

To improve the cost of computing the index $i$, first note that in Algorithm~\ref{alg:fwlp} we only need to compute $\e_i^TA^T\y_k$ for indices $i$
such that there is a possibility that $i = \argmin_j \{c_j - \e_j^TA\y_k\}$. This could be implemented by a data structure to store the
indices $[n]$ in some order so that only the possibly most infeasible indices need to be considered at each iteration. In more detail, suppose $j_k$ indexes the most violated dual constraint on iteration $k$, i.e., $j_k=\argmin_j[\c-A^T\y_k]_j$.  Suppose $j\in[n]$ is some other index.  Based on the value of $[\c-A^T\y_k]_j-[\c-A^T\y_k]_{j_k}$ and prior knowledge of the stepsize (which tends to $0$ with $k$), one knows in advance that constraint $j$ could not be the most violated constraint prior to iteration $k'$, where $k'$ is a computable index satisfying $k'>k$.  Then constraint $j$ does not even have to be considered by the algorithm on all iterations between $k$ and $k'$.

Thus, the algorithm maintains some subset of constraints $\S_k$ on iteration $k$ that need to be examined for possibly being the most violated.  
The computation of $A^T\y_k$ in the primal update thus runs in $\bigO(|\S_{k}|\cdot m)$ plus the cost of updating the data
structure to obtain $\S_{k+1}$.  Naturally, the estimate $\bigO(|\S_{k}|\cdot m)$ is further reduced if $A$ is sparse.

The total cost of iteration $k$ for this implementation of FWLP is thus
\[
  \bigO(n + \lvert \S_k\rvert\cdot m),
\]
plus the cost of updating the proposed data structure and precomputation. For large problems, one would expect that $\lvert\S_k\rvert \ll n$, allowing significant speedup to be achieved
by using the proposed efficient implementation of FWLP.  And, as mentioned earlier, the ``$n$'' term may be dropped with a careful implementation for updating $\x_k$, and the ``$\lvert \S_k\rvert\cdot m$'' term is reduced in the presence of sparsity.

\subsection{FWLP-P} \label{sec:efficient_fwlp}
The iterations of FWLP-P differ from FWLP in that the steps \eqref{eq:rk1} and \eqref{eq:sk1} are more computationally involved.

To solve \eqref{eq:rk1}, we equivalently consider the projection form \eqref{eq:rkproj}, which amounts to solving the quadratic programming problem
\begin{argmini*} 
{\x} { \frac{1}{2}\lVert \w_0 - \x\rVert^2 }{}{} 
\addConstraint{ \e^T\x }{ \leq \xi }{}
\addConstraint{ \x }{ \geq \bz }{},
\end{argmini*}
where $\w_0 = \sqrt{k}(A^T\y_k - \c)$. Expanding, rescaling, and dropping constant terms gives the equivalent problem
\begin{argmini} 
{\x} {-\w^T\x + \frac{1}{2}\lVert\x\rVert^2 }{\label{eq:rk1_qp}}{} 
\addConstraint{ \e^T\x }{ \leq 1 }{}
\addConstraint{ \x }{ \geq \bz }{},
\end{argmini}
where $\w=\w_0/\xi$.
The KKT conditions of \eqref{eq:rk1_qp} are as follows
\begin{align}
  -\w + \x + \mu\e - \z &= \bz,\label{eq:kkt_stat}\\
    \e^T\x &\leq 1,\label{eq:kkt_pf1}\\
    \mu(\e^T\x - 1) &= 0,\label{eq:kkt_cs_p}\\
    \z^T\x &= 0,\label{eq:kkt_cs_d}\\
    \x,\mu,\z &\geq \bz\label{eq:kkt_nn}.
\end{align}
We first prove the following lemma, which will become useful once we define our algorithm for solving \eqref{eq:rk1_qp}.
\begin{lemma} \label{lem:wcases}
  Suppose $(\x,\mu,\z)$ form a KKT solution for \eqref{eq:rk1_qp} with $\e^T\x = 1$.
  There must exist an index $j \in [n]$ such that $w_j > \mu$.
  
  Furthermore, we have the following cases regardless of whether %or not 
  $\e^T\x = 1$.
  \begin{enumerate}[label=(\roman*)]
    \item If $w_j > \mu$ for some $j \in [n]$ , then $x_j > 0$ and $z_j = 0$.
    \item If $w_j < \mu$ for some $j \in [n]$, then $x_j = 0$ and $z_j > 0$.
    \item If $w_j = \mu$ for some $j \in [n]$, then $x_j = z_j = 0$.
  \end{enumerate}
\end{lemma}
\begin{proof}
  By \eqref{eq:kkt_stat}, we have
  \begin{equation} \label{eq:kkt_stat2}
    w_i = x_i - z_i + \mu
  \end{equation}
  for all $i \in [n]$. The condition $\e^T\x = 1$ along with the nonnegativity of $\x$ from
  \eqref{eq:kkt_nn} implies that there must exist an index $j \in [n]$ such that $x_j > 0$.
  Now, the complementarity condition \eqref{eq:kkt_cs_d} along with the nonnegativity of $\z$ from \eqref{eq:kkt_nn}
  imply that $z_j = 0$. It follows from \eqref{eq:kkt_stat2} that $w_j = x_j + \mu > \mu$.

  Now for (i): suppose, $w_j > \mu$ for some $j \in [n]$. By \eqref{eq:kkt_stat2}, $x_j - z_j + \mu > \mu$, which implies $x_j > z_j$. But since
  $z_j \geq 0$ it follows that $x_j > 0$. By applying \eqref{eq:kkt_cs_d} we see that $z_j = 0$.
  The proof for (ii) is analogous.

  For (iii) suppose, $w_j = \mu$ for some $j \in [n]$. By \eqref{eq:kkt_stat2}, $x_j - z_j + \mu = \mu$, which implies $x_j = z_j$.
  By applying \eqref{eq:kkt_cs_d} we must have $x_j = z_j = 0$.
\end{proof}

We now state Algorithm~\ref{alg:rk1} for solving \eqref{eq:rk1_qp}.  Note that the algorithm computes the KKT multipliers $\mu,\z$ as well as $\x$ in order to illustrate its correctness, but in the FWLP-P code, the computation of $\z$ is not needed.
\begin{algorithm}[H]
\begin{algorithmic}[1]
\Require Linear term coefficient $\w \in \R^n$.
\vspace{0.2em}
\State Sort $\w$ and store result in $\bar{\w}$, along with the permutation function $\sigma : [n] \to [n]$
which maps the indices of $\bar{\w}$ back to their original positions in $\w$:
\begin{equation} \label{eq:rk1_sort}
  (\bar{\w}, \sigma) := \text{sort}(\w).
\end{equation}
\State Compute the cumulative sum of $\bar{\w}$:
\begin{equation}
  S := \text{cumsum}(\bar{\w}).
\end{equation}
\For{$j=1,2,\ldots, n$}
    \State Compute
    \begin{equation}
      \mu := \frac{S(j)-1}{j}
    \end{equation}
    \If{$\bar{w}_j \geq \mu$ and either $j = n$, or $\bar{w}_{j+1} \leq \mu$}
      \State Record index $j$ in the variable $J$.
      \Break
    \EndIf
\EndFor
\If{$\mu \geq 0$}
  \State Construct a KKT solution such that $\e^T\x = 1$:
  \begin{equation} \label{eq:construct_kkt_bdry}
    \begin{split}
      %&\x_f = \bar{\w}(1:J) - \mu\cdot\e(1:J),\\
      &\x := \bz,\\
      &\x(\sigma(1:J)) := \bar{\w}(1:J) - \mu\cdot\e(1:J),\\
      %\x_f,\\
      &\z := \mu\e - \w,\\
      &\z(\sigma(1:J)) := \bz.
    \end{split}
  \end{equation}
  \State \Return
\EndIf
\State Otherwise, construct a KKT solution such that $\e^T\x < 1$:
\begin{equation} \label{eq:construct_kkt_int}
  \begin{split}
    &\x := \max(\bz,\w),\\
    &\mu := 0,\\
    &\z := \x - \w.
  \end{split}
\end{equation}
\State \Return
\end{algorithmic}
\caption{An efficient algorithm for finding a KKT solution of \eqref{eq:rk1_qp}.}
\label{alg:rk1}
\end{algorithm}

We now prove that Algorithm~\ref{alg:rk1} is finite and correct.
\begin{theorem}
  Algorithm~\ref{alg:rk1} finds a KKT solution to \eqref{eq:rk1_qp} in finite time.
\end{theorem}
\begin{proof}
  It is clear that the algorithm runs in linear time except for the sorting.

  Consider the case where $(\x,\mu,\z)$ satisfy the KKT conditions with $\e^T\x = 1$. Then by Lemma~\ref{lem:wcases} there must exist
  an index $j\in [n]$ such that $w_j > \mu$, so $J$ must be well-defined on line 6 of the algorithm. Moreover, this means that
  for all $j \in [J], \bar{w}_j \geq \mu_j$, and by applying Lemma~\ref{lem:wcases} we get that $x_j \geq 0$ and $z_j = 0$ for all $j \in [J]$.
  A similar argument tells us that $\bar{w}_j < \mu_j$ and thus $x_j = 0, z_j > 0$ for $j \in J' := [n]\setminus [J]$.
  Using that $\e^T\x = 1$ and that the entries of $\x$ must be zero outside of $[J]$, as well as $z_j = 0$ for all $j \in [J]$, we may multiply \eqref{eq:kkt_stat} by $\e_J^T$ and rewrite to obtain:
  \[
    \e_J^T\w = \e_J^T\x - \e_J^T\z + J\mu = 1 + J\mu,
  \]
  where $\e_J$ is taken to be the vector with ones in entries $[J]$ and zero everywhere else. Rearranging gives the formula for
  $\mu$ used in Algorithm~\ref{alg:rk1}:
  \[
    \mu = \frac{S(J) - 1}{J}
  \]
  where we note that $S(J) = \e_J^T\w$. It follows from assumption of this case $\e^T\x=1$ that $\mu\ge 0$ in line 10 of the algorithm.
  %From \eqref{eq:kkt_nn}, $\mu \geq 0$ must hold and so we go into the if statement on line 10 of the algorithm.
  After equations \eqref{eq:construct_kkt_bdry}, the algorithm has constructed an $\x$ such that
  \[
    \x_J = \w_J - \mu\e_J,\quad \x_{J^{'}} = \bz,
  \]
  and a $\z$ such that
  \[
    \z_J = \bz,\quad \z_{J'} = \mu\e_{J'} - \w_{J'}.
  \]
  It is now easy to check that the $(\x,\mu,\z)$ constructed by the algorithm satisfy the KKT conditions.

  Now consider the case where $(\x,\mu,\z)$ satisfy the KKT conditions with $\e^T\x < 1$. From \eqref{eq:kkt_cs_p},
  $\mu = 0$ and clearly the $\x$ and $\z$ defined in \eqref{eq:construct_kkt_int} satisfy the remaining KKT conditions.
\end{proof}
The cost of Algorithm~\ref{alg:rk1} is bounded below by the cost of either the matrix-vector product $A^T\y_k$ in computing $\w$ or the sort \eqref{eq:rk1_sort},
which take $\bigO(mn)$ and $\bigO(n\log(n))$ operations respectively.
Meanwhile, \eqref{eq:xk1} takes $\bigO(n)$ operations, so overall the two run in $\bigO(n + mn)$ operations if we assume $m \geq \log(n)$.

Algorithm~\ref{alg:rk1} could be sped up by noticing that in order to construct $\x$,%an optimal solution, 
we need only knowledge of the most-violated
dual constraints, i.e., those for which $\c - A^T\y_k$ is most negative. Recall from Section~\ref{sec:efficient_fwlp} that
the computation of $A^T\y_k$ would take $\bigO(\lvert \S_k\rvert \cdot m)$ operations, plus the cost of updating
the data structure storing the indices $\S_k$. The sorting computation cost would also be decreased since only the largest entries would need sorting.

We now consider the dual updates \eqref{eq:sk1} and \eqref{eq:yk1}. Step \eqref{eq:sk1} can be solved by considering the projection form $\proj_{\Gamma}(\sqrt{k}(\b - A\x_{k+1}))$ and noting that this has
a well known closed-form solution \cite[Lemma 6.26]{Beck2017}:
\[
  \s_{k+1}(i) = \max(-\eta,\min(\eta,\sqrt{k}[\b - A\x_{k+1}]_i)),
\]
for each $i \in [m]$. The cost of computing $A\x_{k+1}$ is again $\bigO(\lvert \S_k\rvert\cdot m)$ because we can scale $A\x_{k-1}$ and then add the update $A\r_k/k$.  The number of nonzero entries in $\r_k$ is $|\S_{k+1}|$, which is the number of entries that partake in the projection.
%at $\bigO(mn)$ is the bottleneck in this case
%and 
%we conclude that \eqref{eq:sk1} and \eqref{eq:yk1} can be computed in $\bigO(m + mn)$ time.

\section{Conclusion}
We proposed two primal-dual first-order algorithms, namely FWLP and FWLP-P, for solving linear programming problems
and discussed how both algorithms can be implemented in such a way that significantly improves their efficiency, especially for
large-scale problems.
Our convergence analysis of FWLP-P shows that the algorithm converges to a primal-dual solution with error $\bigO(1/\sqrt{k})$ after
$k$ iterations. Despite this, no convergence proof is known at this time for the simpler and faster algorithm FWLP, and analysis of this algorithm is a topic for
future research.

Another interesting question is how FWLP and FWLP-P can cope with primal or dual infeasibility.  It should be noted that our proof that $U_k\rightarrow 0$ (Theorem~\ref{thm:Ukdec}) did not depend on feasibility nor on the correct choice of $\xi$ and $\eta$, so monitoring $U_k$ cannot diagnose these conditions.  We remark that other first-order algorithms have a theory for infeasibility detection; see, e.g., \cite{applegate2024infeasibility} and \cite{jiang2023range}.  A related question is how the algorithm can detect whether the two parameters $\xi$ and $\eta$ of FWLP and FWLP-P have been chosen correctly.

We showed that FWLP-P converges in the sense that the primal and dual infeasibility measures tend to 0, as does the duality gap.  However, we did not prove convergence of the iterates. In the case that the LP has multiple optimizers, an open question is whether the algorithms converge to a particular optimizer.

\bibliography{fwlp}
\bibliographystyle{plain}
\end{document}